\documentclass[reqno]{amsart}

\usepackage{a4wide,mathrsfs,color}
\usepackage[colorlinks=false]{hyperref}

\title[Length structures on manifolds with continuous Riemannian metrics]{Length structures on manifolds with continuous Riemannian metrics}
\date{\today}

\numberwithin{equation}{section}

\theoremstyle{plain}
\newtheorem{theorem}{Theorem}[section]
\newtheorem{corollary}[theorem]{Corollary}
\newtheorem{proposition}[theorem]{Proposition}
\newtheorem{lemma}[theorem]{Lemma}

\theoremstyle{definition}
\newtheorem{definition}[theorem]{Definition}
\newtheorem{example}[theorem]{Example}

\theoremstyle{remark}
\newtheorem{remark}[theorem]{Remark}

\newcommand{\id}{\operatorname{id}}

\newcommand{\Var}{\operatorname{Var}} 

\def\R{\mathbb{R}}
\def\N{\mathbb{N}}

\def\eps{\varepsilon}
\def\ac{\mathrm{ac}}

\def\rec{\mathrm{rec}}
\def\loc{\mathrm{loc}}

\def\lip{\mathrm{lip}}

\def\A{\mathcal{A}}
\def\cC{\mathcal{C}}
\def\B{\mathcal{B}}

\newcommand\bel[1]{\begin{equation}\label{#1}}
\newcommand\ee{\end{equation}}

\begin{document}
\author[A.\ Burtscher]{Annegret Y.\ Burtscher}
\address{Laboratoire Jacques-Louis Lions, Universit\'e Pierre et Marie Curie (Paris 6), France, and Faculty of Mathematics, University of Vienna, Austria.
Current address: Mathematical Institute, University of Bonn, Endenicher Allee 60, 53115 Bonn, Germany. E-mail: \href{mailto:burtscher@math.uni-bonn.de}{burtscher@math.uni-bonn.de}.} 
\thanks{Appears in \textit{New York Journal of Mathematics} (2015), see \url{http://nyjm.albany.edu/j/2015/21-12.html}.}



\begin{abstract}
 It is well-known that the class of piecewise smooth curves together with a smooth Riemannian metric induces a metric space structure on a manifold. However, little is known about the minimal regularity needed to analyze curves and particularly to study length-mini\-mizing curves where neither classical techniques such as a differentiable exponential map etc.\ are available nor (generalized) curvature bounds are imposed. In this paper we advance low-regularity Riemannian geometry by investigating general length structures on manifolds that are equipped with Riemannian metrics of low regularity. We generalize the length structure by proving that the class of absolutely continuous curves induces the standard metric space structure. The main result states that the arc-length of absolutely continuous curves is the same as the length induced by the metric. For the proof we use techniques from the analysis of metric spaces and employ specific smooth approximations of continuous Riemannian metrics. We thus show that 
when dealing with lengths of curves, the metric approach for low-regularity Riemannnian manifolds is still compatible with standard definitions and can successfully fill in for lack of differentiability.
\end{abstract}

\maketitle


\section{Introduction}

A Riemannian metric on a manifold is needed when considering geometric notions such as lengths of curves, angles, curvature and volumes. For several of these notions it is sufficient to work with the underlying metric space structure induced by the Riemannian metric and a length structure. In this paper we identify the optimal length structure on Riemannian manifolds compatible with the usual metric space structure as the class of absolutely continuous curves and subsequently investigate properties of the length structure of Riemannian manifolds of low regularity. By low regularity we think of Riemannian metrics of regularity less than $\mathcal{C}^{1,1}$. For such metrics the uniqueness of solutions to the geodesic equation is just known to hold~\cite{Har:unique}, but the exponential map is already not a local diffeomorphism anymore but only a bi-Lipschitz homeomorphism~\cite{KSS:exp,Min}. This thinness clearly effects the global structure of low-regularity Riemannian manifolds and it is our aim to seek out other tools 
that can fill in for lack of differentiability. For the purpose of this paper we employ techniques from the analysis of metric spaces.

Low-regularity Riemannian manifolds have already been studied in the literature, particularly in the context of metric geometry. A sequence of closed connected $n$-dimensional Riemannian manifolds $(M_i)_i$ with sectional curvature bounded from below and diameter bounded from above is known to have a subsequence (with respect to the Gromov--Hausdorff distance) converging to a metric space $M$, more precisely, an Alexandrov space with the same lower curvature bound~\cite{GP:manifolds}. Otsu and Shioya~\cite{OS:Alex} showed that an $n$-dimensional Alexandrov space $X$ inherits a $\mathcal{C}^0$-Riemannian structure on $X \setminus S_X$, where $S_X$ denotes the set of singular points in $X$. Whenever $X$ contains no singular points then it is an ordinary $\mathcal{C}^0$-Riemannian manifold. Earlier Berestovskii already showed that locally compact length spaces with curvature bounded from above and below and on which shortest paths can be extended locally are $\mathcal{C}^1$-manifolds with a continuous 
Riemannian metric. This result was improved to show H\"older continuity $\mathcal{C}^{1,\alpha}$ (for $0<\alpha<1$) of the metric components by Nikolaev using parallel translation (for both results see, e.g., \cite{ABN:genRie}). Approximations of such length spaces with curvature bounds by smooth Riemannian manifolds satisfying sectional curvature bounds in aggregate allow one to carry certain results of Riemannian geometry in large over to the metric situation~\cite{Nik:closure1}.
Riemannian manifolds with continuous Riemannian metrics have also been studied by Calabi and Hartman. They showed in \cite{CH:isometries} that isometries in this class of metrics are in general nondifferentiable (unless the Riemannian metrics are, for example, uniformly H\"older continuous).
The distance function on Lipschitz manifolds with Lipschitz Riemannian metrics (and its relations to Finslerian structures) has been studied by De Cecco and Palmieri~\cite{DeP:intdist, DeP:lip} in the 1990s.

With these results in mind we focus in this paper on manifolds with continuous Riemannian metrics. We present results that can be formulated using only the length structure of Riemannian manifolds.

The paper is organized as follows. In Section~\ref{section1}, we initiate our investigations by studying the length structure of manifolds equipped with smooth Riemannian metrics. We recall that the induced length $L_d$ is equal to the arc-length $L$ of curves in the standard setting where the exponential map plays an important role. In Section~\ref{section2}, we extend this result to the class $\A_\ac$ of absolutely continuous curves. We introduce the so-called variational topology on the class of absolutely continuous paths and show that in this topology the piecewise smooth paths are a dense subset. This implies that the class of absolutely continuous curves defines the same length structure as the class of piecewise smooth curves. For the sake of completeness we also show the equivalence of various notions of absolutely continuous curves on Riemannian manifolds used in different contexts in the literature. Finally, in Section~\ref{section3}, we consider manifolds with Riemannian metrics of low regularity.
 We focus here on manifolds that are equipped with a continuous Riemannian metric and whose induced metric therefore still induces the manifold topology. Our first result is that metrics induced by continuous Riemannian metrics are equivalent on compact sets. We proceed by demonstrating that the metric space structure of a manifold induced by a continuous Riemannian metric can be controlled by the metric space structure induced by smooth approximations of the Riemannian metric. In particular, we use such approximations to establish the equivalence of the metric derivative and the analytic derivative. This enables us to prove that $L = L_d$ holds also in the class $\A_\ac$ on any manifold with continuous Riemannian metric. 


\section{Background}
\label{section1}

Let $M$ be a connected smooth manifold endowed with a smooth Riemannian metric $g$, i.e.\ $g_p$ varies smoothly in $p$ on $M$. Let us briefly recall the standard construction to assign a metric $d$ to $M$ via $g$. The class of piecewise smooth curves (with monotonous reparametrizations) on $M$ is denoted by $\A_\infty$. The length of a piecewise smooth curve\footnote{To increase readability, we always assume that an arbitrary path $\gamma\colon [a,b]\to M$ is reparametrized in a way that $a=0$ and $b=1$.} $\gamma\colon [0,1]\to M$ is defined by
\bel{LC1}
 L(\gamma) := \int_0^1 \| \gamma'(t) \|_g \, dt,
\ee
 where $\| v \|_g = \sqrt{g_p(v,v)}$ denotes the norm of $v \in T_pM$ with respect to $g$. The triple $(M,\A_\infty,L)$ defines a \emph{length structure\/} on the topological space $M$. The \emph{intrinsic metric\/} (or distance function) $d = d(g,\A_\infty,L)$ is assigned to $M$ by setting
\bel{d}
 d(p,q) := \inf \{ L(\gamma) \, | \, \gamma \in \A_\infty, \gamma(0) = p, \gamma(1) = q \}, \qquad p,q \in M.
\ee
 It is a standard result from Riemannian geometry that $(M,d)$ defines a metric space structure on $M$ that induces the manifold topology. More precisely, $(M,d)$ is a \emph{length space\/}, i.e.\ a metric space with intrinsic metric.

\medskip
Given an intrinsic metric $d$ it is generally not possible to uniquely reconstruct from knowledge of $d$ alone the length structure from which it was derived.
However, there is a natural way to associate a length structure $L_d$ to a given metric $d$, namely by approximating paths by ``polygons''.

\begin{definition}
 Let $(X,d)$ be a metric space and $\gamma\colon [0,1]\to X$ a (continuous) path in $X$. Then
\bel{Ld}
 L_d(\gamma) := \sup \Big\{ \sum_{i=1}^n d(\gamma(t_{i-1}),\gamma(t_i)) \, \Big| \, n \in \N, 0 = t_0 < t_1 < \ldots < t_n = 1 \Big\}
\ee
 is called the \emph{induced length\/} of $\gamma$.
\end{definition}

 A length structure on any metric space $(X,d)$ is therefore obtained by considering the class $\A_0$ of continuous curves and the induced length function $L_d$. This length structure gives rise to another metric $\hat d$, the intrinsic metric $d(X,\A_0,L_d)$, which is defined analogously to \eqref{d}. In general, $d$ and $\hat d$ do not induce the same topology (for examples see \cite[Ch.\ 1]{Gro:ms} and \cite[Sec.\ 2.3.3]{BBI}). 
 If we start out with a length space $(X,d)$, it is therefore interesting to ask underwhich conditions the original length structure $L$ and the induced length structure $L_d$ coincide. If $L$ and $L_d$ coincide wherever $L$ is defined, then $L_d$ serves as a natural extension of $L$ to $\A_0$.
 
\begin{theorem}\label{thm1}
 Let $M$ be a connected manifold with smooth Riemannian metric $g$. Then
\[
 L(\gamma) = L_d(\gamma), \qquad \gamma \in \A_\infty.
\]
\end{theorem}

\begin{proof}
 ($L_d \leq L$) This is true for any length space, by definition of $L_d$, $d$ and the additivity of $L$.

 ($L \leq L_d$) Let $\gamma\colon [0,1] \to M$ and $t \in (0,1)$ such that $\gamma'(t)$ exists. The exponential map $\exp_{\gamma(t)}$ defines a diffeomorphism on a neighborhood $U$ of $\gamma(t)$. Let $\delta>0$ such that $\gamma([t-\delta,t+\delta]) \subseteq U$. Then,
 \[
\frac{1}{\delta} \, d(\gamma(t),\gamma(t+\delta)) = \frac{1}{\delta} \, \left\|\exp_{\gamma(t)}^{-1}(\gamma(t+\delta))\right\|_{g_{\gamma(t)}} = \left\|\frac{1}{\delta} \, \exp_{\gamma(t)}^{-1}(\gamma(t+\delta))\right\|_{g_{\gamma(t)}},
\] 
 and the metric derivative of $\gamma$ satisfies
 \begin{align}
  \lim_{\delta \rightarrow 0+} \frac{d(\gamma(t),\gamma(t+\delta))}{\delta} &= \left\| \left. \frac{d}{d\delta} \right|_0 \exp^{-1}_{\gamma(t)}(\gamma(t+\delta)) \right\|_{g_{\gamma(t)}} \notag \\
  &= \Big\| (\underbrace{T_0 \exp_{\gamma(t)}}_{\id})^{-1} (\gamma'(t)) \Big\|_{g_{\gamma(t)}} \notag \\
  &= \left\| \gamma'(t) \right\|_{g_{\gamma(t)}}. \label{gammaprime}
 \end{align}
 Moreover,
 \begin{align} \label{eq:proof1}
  \frac{1}{\delta} \, d(\gamma(t),\gamma(t+\delta)) &\leq 
\frac{1}{\delta} \, L_d(\gamma|_{[t,t+\delta]})
\leq \frac{1}{\delta} \int_t^{t+\delta} \| \gamma'(s) \|_{g_{\gamma(t)}} ds
 \end{align}
 by the first part of the proof. Both sides of this inequality converge to $\| \gamma'(t) \|_{g_{\gamma(t)}}$ as $\delta$ tends to $0$. Similarly for $t-\delta$. Note that the intermediate term in \eqref{eq:proof1} may be written as $\frac{1}{\delta} \, L_d(\gamma|_{[t,t+\delta]}) = \frac{1}{\delta} \left( L_d(\gamma|_{[0,t+\delta]}) - L_d(\gamma|_{[0,t]}) \right)$. Thus for almost all $t \in (0,1)$ we obtain that
 \begin{equation*}\label{eq:d1}
  \frac{d}{dt} L_d(\gamma|_{[0,t]}) = \| \gamma'(t) \|_{g_{\gamma(t)}}.
 \end{equation*}
 The fundamental theorem of calculus therefore yields
 \[ L_d(\gamma) = L_d(\gamma) - \underbrace{L_d(\gamma|_{[0,0]})}_{=0} = \int_0^1 \frac{d}{dt} L_d(\gamma|_{[0,t]}) dt = \int_0^1 \| \gamma'(t)\|_{g_{\gamma(t)}} dt = L(\gamma). \qedhere \] 
\end{proof}

\begin{corollary}\label{thm1cor}
 Under the assumptions of Theorem~\ref{thm1},
\[
 d = \hat d,
\]
 where $\hat d = d(M,\A_\infty,L_d)$ is the induced intrinsic metric on $M$. \qed 
\end{corollary}


\section{More general classes of curves on Riemannian manifolds}
\label{section2}

 Let $(M,g)$ again be a connected smooth Riemannian manifold, and $d$ be the distance function induced by the class $\A_\infty$ of piecewise smooth curves on $M$. By Theorem~\ref{thm1}, the length of piecewise smooth curves $\gamma$ is given by 
\[
 L(\gamma) = L_d(\gamma).
\]
 Note that the right hand side of this equation is well-defined for larger classes of curves, indeed $L_d(\gamma)$ makes sense for continuous curves $\gamma$ on $M$. We are therefore interested to also understand the left hand side in more general cases. In particular, we ask for what maximal class $\mathcal{B}$ of curves is $L(\gamma)$ well-defined and equal to $L_d(\gamma)$. We will see that $\mathcal{B}$ is the class of absolutely continuous curves on $M$, and that the metric induced by $\mathcal{B}$ satisfies $d(g,\mathcal{B},L) = d(g,\A_\infty,L)$.


\subsection{Rectifiable curves}

 A continuous path $\gamma\colon I \to M$ is called \emph{rectifiable\/} (or of bounded variation) if $L_d(\gamma) < \infty$. We denote the class of rectifiable curves by $\A_\rec$.

 For $M = \R$ these curves are functions of bounded (pointwise) variation and usually are denoted by $\text{BV}(I)$. Each $\gamma \in \text{BV}(I)$ is differentiable a.e.\ in $I$ and satisfies
\[
 L(\gamma) = \int_I |\gamma'(t)| \, dt \leq \Var_I(\gamma) = L_d(\gamma),
\]
 with equality if and only if $\gamma$ is absolutely continuous on $I$ due to the fundamental theorem of calculus.
 As such, if $(M,g)$ is Euclidean with standard metric, the fact that $d=d_\mathrm{ac}$ is an immediate consequence.
 Since the fundamental theorem of calculus was also required in the proof of Theorem~\ref{thm1} and Corollary~\ref{thm1cor}, the class of absolutely continuous curves is a natural candidate for $\mathcal{B}$. Much of this Section~\ref{section2}, in particular Corollary~\ref{cor1}, can be seen as a generalization of the Euclidean setting to Riemannian manifolds with metrics of low regularity. 

\begin{example} \label{ex:Cantor}
 The Cantor function $\Gamma$ illustrates that absolute continuity is really necessary. Namely, the graph $\gamma = (\id,\Gamma)$ of the Cantor function is a continuous function of bounded variation (and hence differentiable a.e.\ with $\gamma'(t)=(1,0)$), but satisfies
\[
 L(\gamma) = \int_0^1 |\gamma'(t)| \, dt = 1 \neq 2 = \Var_{[0,1]}(\gamma) = L_d(\gamma).
\]
\end{example}


\subsection{Absolutely continuous curves}

 The importance of absolutely continuous functions for geometric questions has already been discovered in the second half of the last century (see \cite{Gro:ms,Pet:rg,Rin:ig} and others). There are various ways to define absolutely continuous curves on $\R$, normed spaces, and metric spaces in general. We will confine ourselves to the most convenient definition for differentiable manifolds, and prove the equivalence to other notions on smooth Riemannian manifolds at the end of this section. Our notion of absolute continuity stems from the standard notion for real-valued functions and a generalization to Banach space-valued functions~\cite{Dub:ac,Kli:rg}.

\begin{definition}\label{def:acR}
 Let $I \subseteq \R$ be an interval (or an open set). A function $f\colon I \to \R^n$ is said to be \emph{absolutely continuous\/} on $I$ (for short, AC on $I$) if for all $\eps>0$ there exists $\delta>0$, such that for any $m \in \N$ and any selection of disjoint intervals $\{(a_i,b_i)\}_{i=1}^m$ with $[a_i,b_i] \subseteq I$, whose overall length is $\sum_{i=1}^m |b_i - a_i| < \delta$, $f$ satisfies
\[
 \sum_{i=1}^m | f(b_i) - f(a_i) | < \eps,
\]
 where $| . |$ (without subindex) denotes the standard Euclidean norm on $\R^n$.

 If $f\colon I \to \R^n$ is absolutely continuous on all closed subintervals $[a,b] \subseteq I$, then it is called \emph{locally absolutely continuous\/} on $I$. 
 The spaces of AC functions and locally AC functions are denoted by $AC(I,\R^n)$ and $AC_\loc(I,\R^n)$, respectively.
\end{definition}

 Note that $AC(I,\R^n) = AC_\loc(I,\R^n)$ if $I$ is a closed interval. The function $f(x) = \frac{1}{x}$, however, is locally absolutely continuous but not absolutely continuous on $(0,1)$.

\begin{definition}\label{def:ac}
 Let $I \subseteq \R$ be a closed interval and $M$ be a connected manifold. A path $\gamma\colon I \to M$ is called \emph{absolutely continuous} on $M$ if for any chart $(u,U)$ of $M$ the composition \[u \circ \gamma \colon
 \gamma^{-1}(\gamma(I) \cap U) \to u(U) \subseteq \R^n\] is locally absolutely continuous.

 The class of absolutely continuous curves on $M$ (with monotonous reparametrizations) is denoted by $\A_\ac$.
\end{definition}

We show that Definitions~\ref{def:acR} and \ref{def:ac} coincide on $\R^n$.

\begin{proposition}
 If $M =  \R^n$, then the notions of absolutely continuous curves in \ref{def:acR} and \ref{def:ac} coincide, i.e.\ $AC(\R^n) = \A_\ac ( \R^n)$.
\end{proposition}

\begin{proof}
 ($AC \subseteq \A_\ac$) Suppose $\gamma \in AC(I,\R^n)$, and let $(u,U)$ be a chart. Since $u$ is smooth it is locally Lipschitz and hence $u \circ \gamma$ is locally absolutely continuous.

 ($\A_\ac \subseteq AC$) Let $\gamma \in \A_\ac$ and $\{ (u_i,U_i) \}_i$ be a set of charts that cover $\gamma(I) \subseteq \R^n$. By definition, all concatenations $\left. u_i \circ \gamma \right|_{\gamma^{-1}(\gamma(I) \cap U_i)}$ are locally absolutely continuous. Again, $u_i^{-1}$ being locally Lipschitz implies that $\gamma_i := \left.\gamma\right|_{\gamma^{-1}(\gamma(I) \cap U_i)} = u_i^{-1} \circ u_i \circ \gamma$ as functions from $\gamma^{-1}(\gamma(I)\cap U_i) \subseteq I$ to $\R^n$ are (componentwise) locally absolutely continuous. Therefore, $\gamma \in AC(I,\R^n)$ by Lemma~\ref{lem1} below.
\end{proof}

\begin{lemma}\label{lem1}
 Let $I \subseteq \R$ be a closed interval, $f\colon I \to \R^n$ be a function and $\mathcal{J} = \{ J_i \}_i$ be an open cover of $I$. If all $\left.f\right|_{J_i}$ are locally absolutely continuous, then $f$ is absolutely continuous on $I$ in the sense of Definition~\ref{def:acR}.
\end{lemma}

\begin{proof}
 Without loss of generality we may assume that $n=1$, i.e.\ $f\colon I \to \R$, and that $\mathcal{J}$ consists of finitely many open intervals $J_i$ of $I$ ($i=1,\ldots,N$). There exists a smooth partition of unity, $\{ \chi_i \}_{i=1}^N$, subordinate to $\mathcal{J}$. For each $i$, in particular, $\chi_i$ is Lipschitz continuous and hence the product $\chi_i f$ absolutely continuous on $I$. Therefore, $\gamma = \sum_{i=1}^N \chi_i f$ is absolutely continuous as well.
\end{proof}

\begin{corollary}
 Let $f\colon I \to \R^n$. Then $f$ is absolutely continuous if and only if $\left. f \right|_J$ is locally absolutely continuous for any open subset $J \subseteq I$. \qed
\end{corollary}

 As in the case of real-valued AC functions, the standard arc-length $L(\gamma)$ of absolutely continuous curves $\gamma$ as defined in \eqref{LC1} is well-defined on manifolds with smooth (or even continuous) Riemannian metrics.

\begin{proposition}\label{prop1}
 Let $M$ be a connected manifold equipped with a continuous\footnote{For the first part of the paper, it is sufficient to consider smooth Riemannian metrics $g$. However, some of the arguments also hold in the more general case of Riemannian metrics that depend only continuously on the points of the manifold and which are considered in detail in Section~\ref{section3}.} Riemannian metric $g$. For any absolutely continuous path $\gamma\colon [0,1] \to M$ the derivative $\gamma'$ exists a.e.\ and $\| \gamma' \|_g \in L^1(I)$. In particular,
\[
 L(\gamma) = \int_0^1 \| \gamma'(t) \|_g \, dt
\]
 is a well-defined length of $\gamma \in \A_\ac$.
\end{proposition}

 In fact, Proposition~\ref{prop1} also holds for bounded Riemannian metrics, but we will not consider such metrics in this paper.

\begin{proof}
 Let $(u,U)$ be a chart on $M$, $u=(x^1,\ldots,x^n)$. By Definition~\ref{def:ac}, each $x^i \circ \gamma\colon \R \supseteq \gamma^{-1}(\gamma(I) \cap U) \to \R$ is locally absolutely continuous. Therefore, all $(x^i \circ \gamma)'$ exist a.e.\ and are locally integrable. Thus
\[
 \| \gamma'\|_g = \sqrt{g(\gamma',\gamma')} = \Big| \sum_{i,j} g_{ij} \, \frac{d(x^i \circ \gamma)}{dt} \, \frac{d(x^j \circ \gamma)}{dt} \Big|^{1/2}
\]
 is well-defined and integrable.
\end{proof}

\begin{corollary}\label{cor2}
 Let $M$ be a connected manifold equipped with a smooth Riemannian metric $g$. If $\gamma \in \A_\ac$, then
\bel{der1}
 \lim_{\delta \to 0} \frac{d(\gamma(t),\gamma(t+\delta))}{\vert \delta \vert} = \| \gamma'(t) \|_g.
\ee
\end{corollary}

\begin{proof}
 This follows from \eqref{gammaprime} by using the exponential map $\exp_{\gamma(t)}$ in the beginning of the second part of the proof for Theorem~\ref{thm1}.
\end{proof}

 The equivalence of the ``metric'' derivative on the left hand side of \eqref{der1} and the analytic derivative on the right hand side will be a crucial step when considering Riemannian metrics of low regularity where the exponential map is not available.


\subsection{Piecewise smooth vs.\ absolutely continuous curves}

 We prove that $L_d(\gamma) = L(\gamma)$ also holds for curves $\gamma \in \A_\ac$. The main steps of our approach are to firstly introduce a new metric on the space $\A_\ac$ and secondly show that $\A_\infty$ is dense in $\A_\ac$ with respect to this "variational topology" on $\A_\ac$. The distance between paths used in our approach below is similar to the distance used in \cite{Mil:mt}, which includes an extra energy term. The denseness of $\A_\infty$ in $\A_\ac$ in turn implies that the intrinsic metric $d_\ac = d(M,\A_\ac,L)$ associated to the class of absolutely continuous curves,
\bel{dac}
 d_\ac(p,q) := \inf \{ L(\gamma) \, | \, \gamma \in \A_\ac, \gamma(0) = p, \gamma(1) = q \}, \qquad p,q \in M,
\ee
 is identical to the standard intrinsic metric $d$ as defined in \eqref{d}. As a result we obtain an extension of Theorem~\ref{thm1} for absolutely continuous curves.

\begin{definition}\label{def:Dac}
 Let $M$ be a connected manifold with continuous Riemannian metric $g$ and induced metric $d$~\eqref{d}. The \emph{variational metric\/} on the class of absolutely continuous paths is defined by
\bel{Dac}
 D_\ac(\gamma,\sigma) := \sup_{t \in I} \, d(\gamma(t),\sigma(t)) + \int_I \big| \|\gamma'(t)\|_g - \|\sigma'(t)\|_g \big| \, dt,
\ee
 for $\gamma, \sigma\colon I \to M$ absolutely continuous paths.
\end{definition}

 Since $(M,d)$ is a metric space\footnote{For continuous Riemannian metrics, this is shown in Proposition~\ref{prop3} below.}, so is $(\A_\ac(M),D_\ac)$.  We call the topology on $\A_\ac$ induced by $D_\ac$ the \emph{variational topology\/} of $\A_\ac$.

 Whenever $(M,d)$ is complete, then $(\A_\ac(M),D_\ac)$ is a complete metric space, too. A proof for this is given later in Section~\ref{ss:ac}.

\begin{lemma}\label{lem2}
 Let $M$ be a connected manifold with continuous Riemannian metric $g$. The length functional $L\colon \A_\ac \to \R$ is Lipschitz continuous with respect to $D_\ac$.
\end{lemma}

\begin{proof}
 For $\gamma,\sigma \in \A_\ac$,
\[
 \left| L(\gamma) - L(\sigma) \right| = \left| \int_0^1 \| \gamma'\|_g - \|\sigma'\|_g \right| \leq \int_0^1 \big| \|\gamma'\|_g - \|\sigma'\|_g \big| \leq D_\ac(\gamma,\sigma). \qedhere
\]
\end{proof}

\begin{theorem}\label{thm2}
 Let $M$ be a connected manifold with continuous Riemannian metric $g$. Then the class $\A_\infty$ of piecewise smooth curves is dense in the class $\A_\ac$ of absolutely continuous curves with respect to the variational topology defined in \ref{def:Dac}.
\end{theorem}

 The idea of the proof is straightforward, but the proof itself is lengthy and technical. On a finite number of chart neighborhoods one approximates the absolutely continuous curve by piecewise smooth curves generated by convolution with mollifiers. Since the end points then do not coincide with the end points of the initial curve, they have to be joined up in a suitable way (namely by sufficiently short curves).

\begin{proof}
 Let $\gamma:[0,1] \rightarrow M$ be a curve in $\A_\ac$. We may cover the image of the curve $\gamma(I)$ by finitely many charts $(u_i,U_i)$ and assume without loss of generality that each $u_i(U_i)$ is convex in $\R^n$ and $U_i \subset\subset M$. Since the set $\overline{\bigcup_i U_i}$ is compact in $M$, the Riemannian norm $\|.\|_g$ can be estimated by a multiple of the Euclidean norm $|.|$ (see proof of Proposition~\ref{prop3}). Without loss of generality we consider them equal in all computations. Furthermore, we pick a partition $0=t_0<t_1< ... < t_N=1$ of $[0,1]$ such that the image of $\left. \gamma \right|_{[t_{j-1},t_j]}$ is contained in one chart $(u_i,U_i)$. We consider a fixed interval $[t_{j-1},t_j]$ and omit the index $i$ from now on.

 Let $\eta>0$. Since $\gamma$ is absolutely continuous, $\| \gamma' \|_g$ is in $L^1_\loc$ by Proposition~\ref{prop1}. By the fundamental theorem of calculus for absolutely continuous functions, and continuity of $\gamma$, there exists $\delta \in (0,\frac{1}{2} \, |t_{j}-t_{j-1}|)$ such that the following inequalities hold:
\begin{subequations}
 \begin{align}
  \sup_{\substack{s,t \in [t_{j-1},t_j]\\|s-t|<2\delta}} d(\gamma(s),\gamma(t)) &< \eta, \label{proof2.b} \\
  \sup_{\substack{s,t \in [t_{j-1},t_j]\\|s-t|<2\delta}} \int_{s}^{t} \|\gamma'\|_g &< \eta, \label{proof2.d} \\ 
  \sup_{\substack{s,t \in [t_{j-1},t_j]\\|s-t|<2\delta}} | u( \gamma(s) ) - u(\gamma(t)) | &< \eta . \label{proof2.a}
 \end{align}
\end{subequations} 

 By convolution with a mollifier $\rho$ we obtain a componentwise regularization of $u \circ \gamma|_{[t_{j-1},t_j]}$. Thus for sufficiently small $\eps>0$ the smooth approximation $\gamma_\eps := u^{-1} ((u \circ \gamma) \ast \rho_\eps) \in \A_\infty$ on $[t_{j-1},t_j]$ satisfies
\begin{subequations}
 \begin{align}
  \sup_{t \in [t_{j-1},t_j]} |u (\gamma(t)) - u (\gamma_\eps(t))| &< \eta, \label{unifconv} \\
  \left\| (u \circ \gamma_\eps)' - (u \circ \gamma)' \right\|_{L^1([t_{j-1},t_j])} &< \eta,  \label{L1locconv}
 \end{align}
\end{subequations}
 and thus on $M$
 \begin{subequations}
 \begin{align}
  \sup_{t \in [t_{j-1},t_j]} d(\gamma(t),\gamma_\eps(t)) &< \eta \label{dconv}, \\
  \int_{t_{j-1}}^{t_j} \left| \| \gamma'\|_g - \| \gamma'_\eps \|_g \right| &< \eta . \label{gintconv}
 \end{align}
\end{subequations}

 Since $u(U)$ is a convex subset of $\R^n$ we can join the points $u(\gamma(t_{j-1}))$ and $u(\gamma_\eps(t_{j-1}+\delta))$ by a straight line $\hat\nu_{j-1}$ in $u(U)$:
\begin{align*}
 &\hat\nu_{j-1}\colon [t_{j-1},t_{j-1}+\delta] \to u(U) \subseteq \R^n, \\
&\hat\nu_{j-1} (t) = u(\gamma(t_{j-1})) + \frac{t-t_{j-1}}{\delta} \left(u(\gamma_\eps(t_{j-1}+\delta)) - u(\gamma(t_{j-1}))\right).
\end{align*}
 Similarly we obtain a straight line $\hat\mu_j$ that connects $u(\gamma_\eps(t_{j}-\delta))$ to $u(\gamma(t_j))$. These straight lines are mapped to (smooth) curves $\nu_{j-1}$ and $\mu_j$ in $M$ by pulling back $\hat\nu_{j-1}$ and $\hat\mu_j$ with $u^{-1}$, respectively.

 Let us compute the lengths of $\nu_{j-1}$ and $\mu_j$. We estimate
\begin{align*}
 \| \nu_{j-1}'(t) \|_g
 &= | \hat\nu_{j-1}'(t) | = \left| \frac{1}{\delta} \left( u(\gamma_\eps(t_{j-1}+\delta)) - u(\gamma(t_{j-1})) \right) \right| \\
  &\leq \frac{1}{\delta} ( \underbrace{| u(\gamma_\eps(t_{j-1}+\delta)) -  u(\gamma(t_{j-1}+\delta))|}_{\substack{< \,\eta \\ \text{for}~\eps~\text{sufficiently small by}~\eqref{unifconv}}} + \underbrace{| u(\gamma(t_{j-1}+\delta)) - u(\gamma(t_{j-1})) |}_{\substack{< \,\eta\\ \text{for}~\delta~\text{sufficiently small by}~\eqref{proof2.a}}} ) < 2 \frac{\eta}{\delta}.
\end{align*}
 Therefore,
\begin{align} \label{lengthnuj-1}
 L(\nu_{j-1}) = \int_{t_{j-1}}^{t_{j-1}+\delta} \| \nu_{j-1}'(t) \|_g \, dt \leq \delta 2 \frac{\eta}{\delta} = 2 \eta,
\end{align}
 and, in a similar fashion, $L(\mu_j) \leq 2 \eta$.

 Each point on $\nu_{j-1}(t)$ and $\mu_j(t)$ is less than $3\eta$ away from the corresponding point $\gamma(t)$ on $\gamma$. For example,  by \eqref{proof2.b} and \eqref{lengthnuj-1} we obtain for $t \in [t_{j-1},t_{j-1}+\delta]$,
\begin{align}\label{gnu}
 d(\gamma(t),\nu_{j-1}(t)) &\leq d(\gamma(t),\gamma(t_{j-1})) + d(\underbrace{\gamma(t_{j-1})}_{\nu_{j-1}(t_{j-1})},\nu_{j-1}(t)) < \eta + 2 \eta = 3 \eta.
\end{align}
 Furthermore we can control the length difference of $\gamma$ and $\nu_{j-1}$ by \eqref{proof2.d} and \eqref{lengthnuj-1}
\begin{align}\label{nuL}
\begin{split}
 \int_{t_{j-1}}^{t_{j-1}+\delta} \left| \|\gamma'\|_g-\|\nu_{j-1}'\|_g \right| &\leq \int_{t_{j-1}}^{t_{j-1}+\delta} \|\gamma'\|_g+ \int_{t_{j-1}}^{t_{j-1}+\delta} \|\nu_{j-1}'\|_g < \eta + 2\eta = 3\eta.
\end{split}
\end{align}

 The same procedure is applied to all $N$ subintervals $[t_{j-1},t_j]$ of $[0,1]$. We choose $\delta$ and $\eps$ sufficiently small so that the above estimates hold on all subintervals. Summing up, we have approximated $\gamma$ on all of $[0,1]$ by a new path $\lambda_\eta$, defined by
\begin{align} \label{lambda_eta}
 \lambda_\eta(t) := \begin{cases}
                             \nu_{j-1}(t) & t\in[t_{j-1},t_{j-1}+\delta] \\
			     \gamma_\eps(t) & t\in[t_{j-1}+\delta,t_{j}-\delta] \\
			     \mu_j(t) & t \in [t_j-\delta,t_j].
                            \end{cases}
\end{align}
 Indeed,
\begin{align} \label{D_ac_eta}
 D_\ac(\gamma,\lambda_\eta) &= \underbrace{\sup_{t \in [0,1]} d(\gamma(t),\lambda_\eta(t))}_{\substack{\leq3\eta\\\textrm{by \eqref{dconv} and \eqref{gnu}}}} + \underbrace{\int_0^1 \left| \|\gamma'\|-\|\lambda_\eta'\| \right|}_{\substack{\leq7N\eta\\\textrm{by \eqref{gintconv} and \eqref{nuL}}}} \leq 10 N \eta
\end{align}
 for any $\eta>0$ (where $N$ is the finite number of open subintervals necessary to cover $[0,1]$ and hence $\gamma([0,1])$ by convex neighborhoods).
\end{proof}

\begin{remark}
 In the case of smooth Riemannian manifolds, the proof of Theorem~\ref{thm2} can be simplified. There, one may cover $\gamma([0,1])$ by a finite number of geodesically convex chart neighborhoods.
 The approximative curves $\gamma_\eps$ are constructed just as in the proof of Theorem~\ref{thm2}, but for the joining curves $\nu_{j-1}$ and $\mu_j$ one may simply use radial geodesics that minimize the distance between two points in a chart neighborhood. Thus any absolutely continuous path can be approximated by a sequence of piecewise smooth paths.
\end{remark}

\begin{corollary}\label{cor1}
 Let $M$ be a connected manifold with continuous Riemannian metric $g$ and with induced distance functions $d$ and $d_\ac$ as defined in \eqref{d} and \eqref{dac}, respectively. Then $d = d_\ac$.
\end{corollary}

\begin{proof}
 Since $\A_\infty \subseteq \A_\ac$, it is clear that $d_\ac \leq d$. On the other hand, $d \leq d_\ac$ follows from the denseness of $\A_\infty$ in $\A_\ac$ with respect to the variational topology.
\end{proof}

 The equality of the induced distance functions $d$ and $d_\ac$ is crucial for answering two questions. We will see that this implies the extension of Theorem~\ref{thm1} to the set $\A_\ac$ (see Section~\ref{ss:lac}). Moreover, we can prove that various notions of absolutely continuous curves on Riemannian manifolds are in fact equal (see Section~\ref{ss:ac} below). 


\subsection{Length structure with respect to absolutely continuous curves} \label{ss:lac}

\begin{lemma}\label{lem3}
 Let $M$ be a manifold with continuous Riemannian metric $g$ and let $\gamma\colon [0,1] \to M$ be absolutely continuous. Then the function $t \mapsto L_d(\gamma|_{[0,t]})$ is absolutely continuous on $[0,1]$.
\end{lemma}

\begin{proof}
 Let $\eps>0$ and let $\{ (t_i,t_{i+1})\}_{i=0}^N$ be disjoint intervals of $[0,1]$. Recall that the length $L_d(\gamma)$ of an absolutely continuous curve $\gamma$ is defined as the variation of $\gamma$ with respect to $d = d_\ac$ (see Definition~\eqref{Ld}). Since $L_d(\gamma|_{[t_i,t_{i+1}]}) \leq \int_{t_i}^{t_{i+1}} \| \gamma' \|_g$ it follows that
\[
 \sum_{i=0}^N \left| L_d(\gamma|_{[0,t_{i+1}]}) - L_d(\gamma|_{[0,t_i]}) \right| \leq \sum_{i=0}^N \left| \int_0^{t_{i+1}}\|\gamma'\|_g - \int_0^{t_i} \|\gamma'\|_g \right|.
\]
 The function $F(t) := \int_0^t \| \gamma' \|_g$ is absolutely continuous on $[0,1]$. Thus if the partition $\{(t_i,t_{i+1})\}_{i=0}^N$ satisfies $\sum_{i=0}^N |t_i - t_{i+1}| < \delta$, then $\sum_{i=0}^N \left| F(t_{i+1}) - F(t_i) \right| < \eps$. Consequently, $t \mapsto L_d(\gamma|_{[0,t]})$ is absolutely continuous as well.
\end{proof}

We are now in a position to prove one of the main results in this section, so far only for manifolds with smooth Riemannian metrics.

\begin{theorem}\label{thm3}
 Let $M$ be a connected manifold with smooth Riemannian metric $g$. Then
\[
 L(\gamma) = L_d(\gamma), \qquad  \gamma \in \A_\ac.
\]
\end{theorem}

\begin{proof}
 For $\gamma: [0,1] \to M$ a piecewise smooth curve this is Theorem~\ref{thm1}. For $\gamma \in \A_\ac \setminus \A$ we have to make some minor adjustments to the proof of Theorem~\ref{thm1}: 
 
 ($L_d \leq L$) is always true since $d=d_\ac$ by Corollary~\ref{cor1}.
 
 ($L \leq L_d$) For all $t\in (0,1)$ such that $\gamma'(t)$ exists (hence almost everywhere) and since $g$ is smooth we obtain in the same way that
\[ \frac{d}{dt} L_d (\gamma|_{[0,t]}) = \| \gamma'(t) \|_g. \]
 By Lemma~\ref{lem3}, $L_d$ is absolutely continuous, and therefore we can also apply the fundamental theorem of calculus which yields
\[ L_d(\gamma) = \int_a^b \frac{d}{dt} L_d (\gamma|_{[0,t]}) \, dt = \int_a^b \| \gamma'(t) \|_g \, dt = L(\gamma). \qedhere \]
\end{proof}


\subsection{Absolutely continuous curves revisited} \label{ss:ac}

 We have seen that the set of absolutely continuous curves $\A_\ac$ induces the same metric structure on a Riemannian manifold $(M,g)$ as the set of piecewise smooth curves $\A_\infty$. Recall that absolute continuity in Definition~\ref{def:ac} was defined locally, in particular, without using the Riemannian or induced metric space structure. We will now see that, since the equality of the length spaces $(M,d)$ and $(M,d_\ac)$ has been shown in Corollary~\ref{cor1}, the definition of absolute continuity used above coincides with the general metric space definition as well as the measure theoretic approach for absolute continuity as, for example, used in \cite{AGS:gf,Vil:ot}.

\begin{definition}\label{def:ac1}
 Let $I \subseteq \R$ be an interval and $(X,d)$ be a metric space. A path $\gamma\colon I \to X$ is called \emph{metric absolutely continuous\/} if for all $\eps>0$ there is a $\delta>0$ so that for any $n \in \N$ and any selection of disjoint intervals $\{ (a_i,b_i) \}_{i=1}^n$ with $[a_i,b_i] \subseteq I$ whose length satisfies $\sum_{i=1}^n |b_i-a_i| < \delta$, $\gamma$ satisfies
\[
 \sum_{i=1}^n d(\gamma(a_i),\gamma(b_i)) < \eps.
\]
 The class of metric absolutely continuous curves on $X$ is denoted by $\B_{\ac}$.
\end{definition}

\begin{definition}\label{def:ac2}
 Let $I \subseteq \R$ be an interval and $(X,d)$ be a metric space. A path $\gamma\colon I \to X$ is called \emph{measure absolutely continuous\/} if there exists a function $l \in L^1(I)$ such that for all intermediate times $a < b$ in $I$,
\[
 d(\gamma(a),\gamma(b)) \leq \int_a^b l(t) \, dt.
\]
 The class of measure absolutely continuous curves on $X$ is denoted by $\cC_{\ac}$.
\end{definition}

 Note that we usually consider $I$ to be a closed interval, and thus do not have to distinguish between local and global concepts of absolute continuity.

\begin{proposition}\label{prop2}
 Let $M$ be a connected manifold with continuous Riemannian metric $g$ and with distance function $d$ induced by the admissible class $\A_\infty$ of piecewise smooth curves. Then all three notions of absolute continuity coincide,
\[
 \A_\ac = \B_\ac = \cC_\ac.
\]
\end{proposition}

\begin{proof}
 ($\cC_\ac \subseteq \B_\ac$) Let $\gamma\colon [0,1] \to M$ be a path in $\cC_\ac$ and $l \in L^1([0,1])$ as in Definition~\ref{def:ac2}. Then $F(s) \colon= \int_0^s l(t) \, dt$ is an absolutely continuous function in $\R$ and therefore, for any subinterval $[a,b] \subseteq [0,1]$, by Definition~\ref{def:ac2},
\begin{align*}
 d(\gamma(a),\gamma(b)) \leq \int_{a}^{b} l(t) \, dt \leq |F(b) - F(a)|.
\end{align*}

 ($\B_\ac \subseteq \A_\ac$) Let $\gamma \in \B_\ac$ and $(u,U)$ be any chart on $M$. Since $u$ is a diffeomorphism, it is Lipschitz on any set $\gamma([a,b]) \subseteq U$. Thus
\[
 \left| u(\gamma(b)) - u(\gamma(a)) \right| \leq C d(\gamma(a),\gamma(b))
\]
 for some constant $C>0$. Since $\gamma$ is absolutely continuous with respect to the induced metric $d$, $u \circ \gamma$ is locally absolutely continuous with respect to the Euclidean norm $| . |$.

 ($\A_\ac \subseteq \cC_\ac$) For $\gamma \in \A_\ac$, consider $l \colon = \| \gamma' \|_g \in L^1$ (see Proposition~\ref{prop1}). Since $d=d_\ac$ by Corollary~\ref{cor1}, it follows that for any $a < b$ in $[0,1]$,
\[
 d(\gamma(a),\gamma(b)) \leq \int_a^b l(t) \, dt. \qedhere
\]
\end{proof}

 We are now in a position to prove completeness of the metric space $(\A_\ac(M),D_\ac)$ introduced above.
 
\begin{proposition}
 Let $M$ be a connected manifold with a continuous Riemannian metric $g$. If $(M,d)$ is complete as a metric space, then so is the space $\A_\ac(M)$ of absolutely continuous paths together with the variational metric $D_\ac$ introduced in Definition~\ref{def:Dac}.
\end{proposition}

\begin{proof}
 Since $(M,d)$ is a metric space, so is $(\A_\ac(M),D_\ac)$. Suppose $(\gamma_n)_n$ is a Cauchy sequence of absolutely continuous paths $\gamma_n \colon I \to M$ with respect to the variational topology given by $D_\ac$. Since $(M,d)$ is complete and because of uniform convergence of continuous curves, the pointwise defined limit
 \[
 \gamma(t) := \lim_{n \to \infty} \gamma_n(t)
 \]
 is a continuous curve in $M$. Moreover, for any $a,b \in I$,
 \[
 d(\gamma(a),\gamma(b)) = \lim_{n \to \infty} d(\gamma_n(a),\gamma_n(b)) \leq \lim_{n \to \infty} \int_a^b \| \gamma_n'(t) \|_g \, dt = \int_a^b \lim_{n \to \infty} \| \gamma_n'(t) \|_g \, dt,
 \]
 thus by Definition~\ref{def:ac2} with $l = \lim_{n \to \infty} \| \gamma_n' \|_g \in L^1 (I)$ the limiting curve $\gamma$ is absolutely continuous. 
\end{proof}


\subsection{Relations between classes of curves}

 Let $\A_1$ and $\A_0$ denote the class of piecewise $\mathcal{C}^1$ and continuous curves, respectively, and let $\A_\lip$ denote the class of Lipschitz curves and $\A_{H^1}$ the class of $H^1$ curves (see, e.g., \cite[Sec.\ 2.3]{Kli:rg}).
 On any differentiable manifold, the following inclusions hold
\[
 \A_\infty \subseteq \A_1 \subseteq \A_\lip \subseteq \A_{H^1} \subseteq \A_\ac \subseteq \A_\rec \subseteq \A_0.
\]
 We have seen that the class $\A_\ac$ of absolutely continuous curves on a Riemannian manifold induce the same metric space structure as $\A_\infty$, thus so do $\A_1$, $\A_\lip$ and $\A_{H^1}$. In fact, by \cite[Lem.~1.1.4]{AGS:gf}, absolutely continuous curves are just Lipschitz curves if one uses reparametrizations that are increasing and absolutely continuous. The length of rectifiable curves $\A_\rec$ and continuous curves $\A_0$ can only be defined if a metric space structure is already present, since the arc-length definition is not meaningful in this setting (recall Example~\ref{ex:Cantor}).


\section{Manifolds with continuous Riemannian metrics}
\label{section3}

 In Proposition~\ref{prop1} we proved that the arc-length of absolutely continuous curves is well-defined even if the Riemannian metric $g$ is only continuous (or even bounded). In this case, however, we do not have the usual tools of Riemannian geometry at hand, e.g.\ geodesic equations, the exponential map, curvature and so on. Despite this handicap we will see that the metric space structure of such Riemannian manifolds of low regularity is not so different from those of smooth Riemannian manifolds. By gathering and introducing new tools, we will also see that Theorem~\ref{thm1} holds in the low-regularity situation.

 Throughout we call any pointwise defined positive definite, symmetric $(0,2)$-tensor field $g$ on a differentiable connected manifold $M$ a \emph{Riemannian metric\/}. In other words, $g$ should be seen as a positive definite symmetric tensor field that is not necessarily smooth. Moreover, we assume that the arc-length $L(\gamma)$ for $\gamma \in \A_\infty$ is well-defined for $g$. We will motivate why continuity of $g$ is desirable.


\subsection{Metric space structure}
\label{subsection3.1}

 Suppose that $g$ is a Riemannian metric of low regularity. Let $\A_\ac$, $L$ and $d = d(g,\A_\ac,L)$ be defined as in Section~\ref{section1}. The triple $(M,\A_\ac,L)$ should define a length structure on $M$ and induce a metric $d$ on $M$. From Definition~\eqref{d} we immediately deduce that $d$ is symmetric, nonnegative and satisfies the triangle inequality. Thus $d$ is a pseudo-metric on $M$. Continuity of the Riemannian metric implies more.

\begin{proposition}\label{prop3}
 Let $M$ be a connected manifold with continuous Riemannian metric $g$. The following properties hold:
\begin{enumerate}
 \item $(M,\A_\ac,L)$ is an admissible length structure on $M$, that is, the length of paths is additive, continuous on segments, invariant under reparametrizations and agrees with the topology on $M$ in the sense that for a neighborhood $U$ of a point $p$ the length of paths connecting $p$ with the complement of $U$ is bounded away from $0$.
 \item The distance function $d$ as defined in \eqref{d} induces the manifold topology on $M$.
\end{enumerate}
 In particular, $(M,d)$ is a length space.
\end{proposition}

\begin{proof}
 (i) The class $\A_\ac$ of absolutely continuous paths is closed under restrictions, concatenations and monotonous reparametrizations. The length $L$ of paths,
\[
 L(\gamma) = \int_I \| \gamma' \|_g, \qquad \gamma \in \A_\ac,
\]
 is additive, continuous on segments and invariant under reparametrizations.

 It remains to be shown that the length structure is compatible with the topology on $M$. Let $p \in M$ and $(u,U)$ be a chart of $M$ at $p$ with $u(p)=0$, $u=(x^1,...,x^n)$. Pick $r>0$ such that $\overline{B_r(0)} = \{ x \in \R^n \, | \, |x| \leq r \} \subseteq u(U)$ and let $K := u^{-1}(\overline{B_r(0)})$. The Euclidean metric with respect to $U$ is given by
\[
 e_U := \delta_{ij} \, dx^i \otimes dx^j.
\]
 Both metrics $g$ and $e_U$ are nondegenerate and thus induce isomorphisms $TM \to T^*M$. Let $\eta_1,...,\eta_n$ denote the (positive) eigenvalues and $v_1,...,v_n$ the eigenvectors of the symmetric tensor $e_U^{-1} \circ g$, i.e.\
\[
 g(v_i,.) = \eta_i \, e_U(v_i,.).
\]
 By $\underline{\eta}, \overline{\eta}$ we denote the smallest and largest eigenvalues, respectively, and define $\lambda,\mu \colon U \to \R^+$ by $\lambda := \sqrt{n \underline{\eta}}$ and $\mu := \sqrt{n \overline{\eta}}$. Thus for all $q \in U$, $v \in T_qM$,
\begin{align*}
 \lambda(q) \, \|v\|_{e_U} \leq \|v\|_g \leq \mu(q) \, \|v\|_{e_U}.
\end{align*}
 Since $g$ is a continuous Riemannian metric, the functions $\lambda, \mu \colon U \to \R^+$ are continuous, and thus on the compact set $K$,
\[
 \lambda_0 := \min_K \lambda > 0 \qquad \mathrm{and} \qquad \mu_0 := \max_K \mu < \infty.
\]
 Therefore,
\begin{align}\label{l0m0}
 \lambda_0 \, \|v\|_{e_U} \leq \|v\|_g \leq \mu_0 \, \|v\|_{e_U}, \qquad v \in T_q M, q \in K.
\end{align}
 Let $y \in M \setminus \operatorname{int} (K)$ be connected to $p$ by a path $\gamma \colon I \to M$. Choose $t_0 \in I $ such that $\gamma(t_0) \in \partial K \cap \gamma(I)$ then
\[
 0 < r \lambda_0 = \lambda_0 | (u \circ \gamma) (t_0) | \leq \int_0^{t_0} \lambda_0 \| \gamma'(t) \|_{e_U} \, dt \stackrel{\eqref{l0m0}}{\leq} \int_0^{t_0} \| \gamma'(t) \|_g \, dt \leq L_g(\gamma), 
\]
 where the first inequality is due to the fact that $d=d_\mathrm{ac}$ in Euclidean space.
 Hence the length of paths connecting $p$ with $M \setminus \operatorname{int} (K)$ is bounded away from $0$.

 (ii) The calculation in (i) implies that $(M,d)$ is a metric space. A short calculation based on \eqref{l0m0} implies that $d$---since locally Euclidean---also induces the topology of $M$ (cf.\ the proof of \cite[Ch.\ 5.3, Thm.\ 12]{Pet:rg}).
\end{proof}

 The following examples further motivate the study of manifolds that are equipped with continuous Riemannian metrics.

\begin{example}
 Alexandrov spaces with curvature bounded from below (and above) are connected complete locally compact length spaces with finite Hausdorff dimension that satisfy a triangle comparison condition to obtain a notion of curvature bounds (see, e.g., \cite{ABN:genRie,BGP:axcbb}). They were introduced as generalizations of Riemannian manifolds with sectional curvature bounds by turning Toponogov's theorem into a definition. Otsu and Shioya~\cite{OS:Alex} proved that an $n$-dimensional Alexandrov space $X$ of curvature bounded from below carries---minus some singular points---a $\mathcal{C}^0$-Riemannian structure. $X$ is a $\mathcal{C}^0$-Riemannian manifold in the ordinary sense whenever $X$ contains no singular points.
\end{example}

\begin{example}
 A Busemann G-space is a finitely compact metric space with an intrinsic metric, in which the conditions of local prolongability of segments and of the nonoverlapping of segments are satisfied. Such spaces are geodesically complete. G-spaces satisfying an additional axiom related to boundedness of curvature generalize complete Riemannian manifolds of class $\mathcal{C}^k$ for $k \geq 2$. It can be shown that they are in fact $\mathcal{C}^1$-manifolds with continuous Riemannian metrics \cite{Ber:BusG}.
\end{example}

In general, different Riemannian metrics will not induce the same length structure. However, we can prove equivalence on compact sets for metrics induced by continuous Riemannian structures.

\begin{lemma}\label{lem4}
 Let $M$ be a compact connected manifold and $g$ and $h$ two continuous Riemannian metrics on $M$. Then there exist constants $c,C>0$ such that
\[
 c \, d_h(p,q) \leq d_g(p,q) \leq C \, d_h(p,q), \qquad p,q \in M.
\]
\end{lemma}

\begin{proof}
 By \eqref{l0m0} there exist positive functions $\lambda^g_0, \mu^g_0$ and $\lambda^h_0, \mu^h_0$ on $M$ such that 
\begin{align}\label{ghg}
 \lambda^g_0 \lambda^h_0 \| v \|_h \leq \lambda^h_0 \mu^h_0 \| v \|_g \leq \mu^g_0 \mu^h_0 \| v \|_h.
\end{align}
 Compactness of $M$ furthermore implies the existence of length-minimizing curves with respect to $d_g$ and $d_h$ \cite[Lem.\ 1.12]{Gro:ms}, and thus concludes the proof.
\end{proof}

 The noncompact analogue of Lemma~\ref{lem4}---even on compact subsets---is more involved because in the general situation little is known about the existence of (locally) length-minimizing curves. We now show that paths whose length approximates the distance $d$ are trapped in certain neighborhoods $U$ of compact subsets $K$. There is no need to even assume geodesic completeness.

\begin{theorem}\label{thm4}
 Let $M$ be a connected manifold and $g$ and $h$ two continuous Riemannian metrics on $M$. Then on every compact set $K$ in $M$, there exist constants $c,C>0$ such that
\[
 c \, d_h(p,q) \leq d_g(p,q) \leq C \, d_h(p,q), \qquad p,q \in K.
\]
\end{theorem}

\begin{proof}
 We argue by contradiction. Suppose that for all $n \in \N$ there are $p_n, q_n \in K$ such that
\begin{align}\label{gh:n}
 d_g(p_n,q_n) > n \, d_h(p_n,q_n).
\end{align}
 By passing on to subsequences we can assume that $(p_n)_n$ and $(q_n)_n$ convergence to the same $p \in K$.

 Let us analyze neighborhoods of $p$. Open balls at $p \in M$ with radius $r>0$ with respect to the distance function $d_h$ induced by the Riemannian metric $h$ will be denoted by
\[
 B^h_r(p) := \{ q \in M \, | \, d_h(p,q) < r \} .
\]
 Since $M$ is locally compact and, by Proposition~\ref{prop3}, $d_h$ induces the manifold topology, there exists an $r_0>0$ such that $\overline{B^h_{r_0}(p)}$ is compact. Let $r = \frac{r_0}{4}$ and $x,y \in B^h_r(p)$. By definition~\eqref{d}, for all $\eps \in (0,r)$, there exists a path $\gamma_\eps \colon [0,1] \to M$ in $\A_1$ between $x$ and $y$ such that
\[
 L(\gamma_\eps) < d_h(x,y) + \eps.
\]
 These paths $\gamma_\eps$ are mapped to the open neighborhood $B^h_{r_0}(p)$ of $p$ since for $t \in [0,1]$,
\begin{align*}
 d_h(p,\gamma_\eps(t)) &\leq d_h(p,x) + d_h(x,\gamma_\eps(t))
 \leq d_h(p,x) + L_h(\gamma_\eps) \\
&\leq d_h(p,x) + d_h(x,y) + \eps < 4r = r_0
\end{align*}
 Continuity of $g$ and $h$ imply that the same inequalities as in \eqref{ghg} hold on $\overline{B^h_{r_0}(p)}$. In particular, for $C = \frac{\mu^g_0}{\lambda^h_0} > 0$,
\[
 d_g(x,y) \leq L_g(\gamma_\eps) \leq C L_h(\gamma_\eps) < C d_h(x,y) + C \eps.
\]
 Since $C$ is independent of $\eps$ as well as $x$ and $y$, this implies 
\begin{align}\label{gh}
 d_g(x,y) \leq C \, d_h(x,y), \qquad x,y \in B^h_r(p).
\end{align}
 By construction both sequences $(p_n)_n$ and $(q_n)_n$ converge to $p$ and thus, by \eqref{gh}, for sufficiently large $n$,
\[
 d_g(p_n,q_n) \leq C \, d_h(p_n,q_n).
\]
 This contradicts \eqref{gh:n}.
\end{proof}

\begin{remark}
 Theorem~\ref{thm4} is a very specific property of Riemannian manifolds. It does not hold on arbitrary (compact) metric spaces. For example, on the interval $[0,1]$ the two metrics
\[
 d_1(x,y) = |x-y| \qquad \text{and} \qquad d_2(x,y) = \sqrt{|x-y|}
\]
 induce the same topology. By inserting $x=0$ and $y_n = \frac{1}{n}$, however, it becomes clear that $d_1$ and $d_2$ are not metrically equivalent since
\[
 \lim_{n \to \infty} \frac{d_2(x,y_n)}{d_1(x,y_n)} = \infty.
\]
\end{remark}


\subsection{Length structure on manifolds with continuous Riemannian metrics}

 On a manifold $M$ with smooth Riemannian metric $g$, the exponential map is a local diffeomorphism. This fact has been used in the proof of Theorem~\ref{thm1} and Theorem~\ref{thm3} to conclude that $L=L_d$ for both, the class $\A_\infty$ of piecewise smooth curves as well as the class $\A_\ac$ of absolutely continuous curves. On manifolds with a Riemannian metric of low regularity, more precisely below $\cC^{1,1}$, the geodesic equation cannot be solved uniquely~\cite{Har:unique} and thus the exponential map is not available. We will see that we can make use of the metric space structure in such general situations.

\begin{definition}
 Let $(X,d)$ be a metric space. For any path $\gamma \colon I \to X$ we denote the \emph{metric derivative\/} of $\gamma$ by
\begin{align}\label{metricderivative}
 |\dot\gamma|(t) := \lim_{\delta\to0} \frac{d(\gamma(t+\delta),\gamma(t))}{|\delta|}, \qquad t \in I,
\end{align}
 whenever the limit exists.
\end{definition} 

This is the quantity that arises in the proof of Theorem~\ref{thm1}.

 \begin{lemma}\label{lem5}
  Let $M$ be a connected manifold with continuous Riemannian metric $g$ and let $\gamma\colon I \to M$ be an absolutely continuous path in $M$. Then the metric derivative $|\dot\gamma|$ of $\gamma$ exists a.e.\ on $I$. Moreover, it is the minimal $L^1(I)$ function in the Definition~\ref{def:ac2} of $\cC_\ac$, i.e.\ if $l \in L^1(I)$ satisfies
\[
 d(\gamma(a),\gamma(b)) \leq \int_a^b l(t) \, dt, \qquad a<b ~\text{in}~ I,
\]
 then $|\dot\gamma| \leq l$ a.e.
 \end{lemma}

\begin{proof}
 Lemma~\ref{lem5} holds in any metric space with absolutely continuous curves (and $AC^p$ curves, $p \in [1,\infty]$) as defined in Definition~\ref{def:ac2}. A proof may be found in \cite[Thm.\ 1.1.2]{AGS:gf}. For further reference, we recall the main arguments:

 Since $\gamma$ is continuous, $\gamma(I)$ is compact and hence separable. Thus we may choose a dense sequence $(x_n)_n$ in $\gamma(I)$. The functions
\[ 
 \varphi_n (t) := d(\gamma(t),x_n), \qquad n \in \N, t \in I
\]
 are absolutely continuous, and therefore their derivatives $\varphi_n'$ exist a.e.\ and are integrable. Since countable unions of null sets are null, the function
\begin{align}\label{phi}
 \varphi(t) := \sup_{n \in \N} |\varphi_n'(t)|
\end{align}
 is a.e.\ absolutely continuous and integrable by the Lemma of Fatou. It can be shown, that, whenever $\varphi$ and the metric derivative of $\gamma$ exist,
\begin{align}\label{mg}
 \varphi(t)  = |\dot\gamma|(t), \qquad t \in I ~\text{a.e.}
\end{align}
 holds.
\end{proof}

\begin{proposition}\label{prop5}
 Let $M$ be a connected manifold with continuous Riemannian metric $g$. Then there exists a sequence of smooth Riemannian metrics $(g_n)_n$ such that
\begin{enumerate}
 \item $g_n$ converge uniformly to $g$, and
 \item the induced distance functions $d_n$ converge uniformly to $d$ on $M$.
\end{enumerate}
\end{proposition}

\begin{proof}
 (i) Let $p \in M$ and $K_p$ a compact neighborhood of $p$.  By convolution with mollifiers, we can locally approximate $g$ by a sequence of smooth Riemannian metrics $h^p_n$ (positive definiteness is an open condition, symmetry can be obtained by construction). By using the same arguments as in the proof of Proposition~\ref{prop3}, that is by considering the eigenvalues of $g^{-1} \circ h^p_n$ (all of which are converging to $1$ uniformly on $K_p$) and by restricting ourselves to a subsequence of $(h^p_n)_n$, we may assume that
\begin{align*}
 \frac{n-1}{n} \| v \|_g \leq \| v \|_{h^p_n} \leq \frac{n+1}{n} \| v \|_g, \qquad v \in T_qM, q \in K_p.
\end{align*}
 Since $(M,d)$ is a metric space, $M$ is paracompact. A partition of unity $\{ \alpha_p \}_{p \in M}$ subordinate to the cover $\{ \operatorname{int} (K_p) \}_{p \in M}$ is used to patch these local approximations $h^p_n$ of $g$ together and obtain a sequence of approximating smooth Riemannian metrics $g_n := \sum_{p \in M} \alpha_p h^p_n$ on $M$ satisfying the above estimate globally, i.e.\ 
\begin{align} \label{eq:gn}
 \frac{n-1}{n} \| v \|_g \leq \| v \|_{g_n} \leq \frac{n+1}{n} \| v \|_g, \qquad v \in T_pM, p \in M. 
\end{align}
 (ii) Let $p,q \in M$. By definition of $d$, for every $\eps>0$ there exist curves $\gamma^\eps$ between $p$ and $q$ satisfying $L(\gamma^\eps) < d(p,q) + \eps$, and thus by \eqref{eq:gn},
\[
 d(p,q) + \eps > L(\gamma^\eps) \geq  \frac{n}{n+1} L_n(\gamma^\eps) \geq \frac{n}{n+1} d_n(p,q), 
\]
 hence
\[
 \frac{n+1}{n} d(p,q) \geq d_n(p,q), \qquad p,q \in M.
\]
 Similarly, for every $d_n$ we have curves $\gamma_n^\eps$ satisfying $L_n(\gamma_n^\eps) < d_n(p,q) + \eps$, which leads to
\[
  \frac{n-1}{n} d(p,q) \leq  d_n(p,q), \qquad p,q \in M. \qedhere
\]
\end{proof}

In the proof of Theorem~\ref{thm1} it was essential to show equivalence of the derivative and the metric derivative of paths. The application of the exponential map was used for determining this equivalence, however in the low regularity setting we do not have this tool at hand. We now establish the same result for manifolds with continuous Riemannian metrics by combining the metric and analytic world.

\begin{proposition}\label{prop4}
 Let $M$ be a manifold with continuous Riemannian metric $g$. For any absolutely continuous path $\gamma \in \A_\ac$, $\gamma \colon I \to M$, the analytic and metric derivatives coincide, i.e.
\[
 \| \gamma'(t) \|_g = | \dot\gamma | (t), \qquad t \in I ~\text{a.e.}
\]
\end{proposition}

This is Corollary~\ref{cor2} for smooth Riemannian metrics and whenever the exponential map is a local diffeomorphisms. For general continuous Riemannian metrics, the ($\geq$) inequality is also easy to see. By Proposition~\ref{prop1}, $\| \gamma'\|_g \in L^1(I)$, and since $d=d_\ac$ by Corollary~\ref{cor1},
\[
 d(\gamma(a),\gamma(b)) \leq \int_a^b \| \gamma'(t) \|_g \, dt, \qquad a < b ~\text{in}~ I.
\]
 By Lemma~\ref{lem5}, $| \dot\gamma|$ is the smallest such $L^1$-function. The proof for equality makes use of the convergence properties obtained in Proposition~\ref{prop5}.

\begin{proof}
By Proposition~\ref{prop5} we can approximate $g$ by smooth metrics $g_n$ such that $d_n$ converges to $d$ uniformly on $M$. Moreover, by Theorem~\ref{thm1}, the metric derivative with respect to $d_n$ exists and equals the norm of the analytic derivative,
\[
 \lim_{\delta \to 0} \frac{d_n(\gamma(t+\delta),\gamma(t))}{|\delta|} = \| \gamma'(t) \|_{g_n}.
\]
 Therefore we may interchange the limits, and by also make use of the convergence obtained in (i) and (ii) of Proposition~\ref{prop5}, obtain
\[
 | \dot\gamma | (t) = \lim_{\delta \to 0} \lim_{n \to \infty} \frac{d_n (\gamma(t+\delta),\gamma(t))}{| \delta |} = \lim_{n \to \infty} | \dot\gamma |_n (t) = \lim_{n \to \infty} \| \gamma'(t) \|_{g_n} = \| \gamma'(t) \|_g. \qedhere
\]
\end{proof}

 Proposition~\ref{prop4} puts the tools out for proving the main result of this section. We show equality of the arc-length $L$ and the induced length $L_d$ as defined in \eqref{LC1} and \eqref{Ld}, respectively, for absolutely continuous curves on manifolds equipped with continuous Riemannian metrics. The proof makes use of techniques that arise from the metric space structure $(M,d)$ only.

 Let us define the \emph{metric arc-length\/} $\widetilde{L}$ of a path $\gamma\colon I \to M$ by
\[
 \widetilde{L}(\gamma) := \int_I |\dot\gamma|(t) \, dt.
\]

\begin{theorem}\label{thm5}
 Let $M$ be a manifold with continuous Riemannian metric $g$. Then
\[
 L(\gamma) = L_d(\gamma) = \widetilde{L}(\gamma), \qquad \gamma \in \A_\ac.
\]
\end{theorem}

\begin{proof}
 ($L_d = \widetilde{L}$) 
 Let $\gamma\colon [0,1] \to M$ be an absolutely continuous curves. Consider $\varphi$ as defined in \eqref{phi}. Together with Fatou's Lemma, equality \eqref{mg} implies that
\begin{align}\label{Ld:gammadot}
  L_d(\gamma) = \int_0^1 |\dot\gamma|(t) \, dt.
\end{align}
 holds for all absolutely continuous paths (cf.\ \cite[Thm.\ 4.1.6]{AT:metric}).

 ($L = \widetilde{L}$) By Proposition~\ref{prop4}, $|\dot\gamma| = \| \gamma'\|_g$ for a.e.\ $t \in [0,1]$.
\end{proof}


\section*{Acknowledgments}

The author would like to thank J.D.E.\ Grant and M.\ Kunzinger for several discussions on the topic. This research was supported by a "For Women in Science" fellowship by L'Or\'eal Austria, the Austrian commission of UNESCO and the Austrian Ministry of Science and Research, and by the Austrian Science Fund in the framework of project P23714. 


\end{document}